\documentclass[11pt]{amsart}
\usepackage[mathscr]{eucal}
\usepackage{times}
\usepackage{amsmath,amssymb,latexsym,amscd}   
\usepackage{hyperref}
\hypersetup{colorlinks=true,linkcolor=blue,citecolor=red,linktocpage=true}
\usepackage[utf8]{inputenc}
\usepackage[T1]{fontenc}
\usepackage{graphicx}
\usepackage[all,cmtip]{xy}
\usepackage{fancyhdr}
\usepackage{url}
\usepackage{mathrsfs}
\usepackage{upgreek}
\usepackage{enumerate}
\usepackage{soul}

\DeclareMathOperator{\ann}{ann}

\DeclareMathOperator{\Ann}{Ann}

\DeclareMathOperator{\pt}{pt}

\DeclareMathOperator{\mx}{Max}

\DeclareMathOperator{\Hom}{Hom}
\DeclareMathOperator{\End}{End}

\newcommand{\V}{\mathcal{V}}
\newcommand{\U}{\mathcal{U}}

\newcommand{\ltc}{\mathcal{L}} 

\theoremstyle{plain}
\newtheorem*{theorem*}{Theorem}
\newtheorem{thm}{Theorem}[section]
\newtheorem{cor}[thm]{Corollary}
\newtheorem{lem}[thm]{Lemma}
\newtheorem{prop}[thm]{Proposition}

\theoremstyle{definition}
\newtheorem{dfn}[thm]{Definition}
\newtheorem{obs}[thm]{Remark}
\newtheorem{ej}[thm]{Example}

\newcommand{\sm}{\sigma[M]}

\begin{document}

\title[On the de Morgan's laws for modules
]{On the de Morgan's laws for modules}


\author[Medina,\\Sandoval,\\ Zald\'ivar]{Mauricio Medina-B\'arcenas\\ Martha Lizbeth Shaid Sandoval-Miranda\\ \'Angel Zald\'ivar-Corichi }

\begin{abstract} In this investigation we give a module-theoretic counterpart of the well known De Morgan's laws for rings and topological spaces. We observe that the module-theoretic De Morgan's laws are related with semiprime modules and modules in which the annihilator of any fully invariant submodule is a direct summand. Also, we give a general treatment of De Morgan's laws for ordered structures (idiomatic-quantales). At the end, the manuscript goes back to the ring theoretic realm, in this case we study the non-commutative counterpart of Dedekind domains, and we describe Asano prime rings using the strong De Morgan law.
\end{abstract}

\maketitle 

\begin{flushleft}

\end{flushleft}

\section{Introduction}\label{sec1}
The uses of ordered structures in the study of algebraic entities such as rings, categories and topological spaces goes back to \cite{ward1939residuated} and \cite{dowker1966quotient}. For example
frames (locales, complete Heyting algebras) can be seen as  algebraic ordered structures of topological spaces. In fact,  a decent analysis of topological spaces can be done via frame theory \cite{johnstone1986stone,picado2012frames}. Also frames can be understood as algebraic models of infinitary logic (every complete Boolean algebra is a frame).
In the ring theoretic side, the theory of quantales \cite{mulvey1986suppl} offers an amalgam to the study of rings via the quantale of (left, right) ideals, see for example \cite{rosenthal1990quantales,niefield1985note}.

Like in the case of topological spaces,  the behavior of a module can be done via its lattice of submodules as it has been explored in \cite{albu2014topics,albu2004modular,albu2014osofsky,simmonsvarious,castro2019some}. 

In \cite{BicanPr}, the authors introduce a product of submodules of a given module. Under some projectivity conditions, this product behaves well with the ordered structure in the set of submodules and the lattice of submodules becomes a quantale. The study of this quantale can be found in the  authors recent papers \cite{mauquantale,medina2018attaching,medina2018strongly}.

As we mentioned before, the connection with certain aspects of logic cannot be avoided. For example, in \cite{johnstone1979conditions}, the author shows (as a consequence of a more general result) that, for a topological space $S$, its topology $\mathcal{O}(S)$ satisfies the De Morgan's Law, that is, 
\[\neg(A\wedge B)=\neg A\vee \neg B\]
if and only if $S$ is extremely disconnected. The frame $\mathcal{O}(S)$ satisfies 
\[(A\Rightarrow B)\vee(B\Rightarrow A)=S\]
if and only if every closed subspace of $S$ is extremely disconnected (where $\Rightarrow$ denotes the implication and $\neg$ the negation in the frame $\mathcal{O}(S)$). 

In \cite{niefield1985strong,niefield1985note,niefield1995algebraic}, it is given an explanation of the De Morgan's Laws in the ring-theoretic setting and the similarities with those characterizations obtained in the topological case.

The aim of the present manuscript is to answer the natural question: {\it What characterizations of modules can be obtained in the presence of the De Morgan's laws on the quantale of submodules?} We give answers to this question and other related topics.

In \cite[Theorem 2.4]{niefield1985strong} a commutative Noetherian Dedekind domain is characterized as a Noetherian domain satisfying the strong algebraic de Morgan's law, that is, given a commutative Noetherian domain $R$, $R$ is a Dedekind domain if and only if $R=(I:J)+(J:I)$ for all ideals $I,J$ if and only if the lattice of ideals of $R_P$ is totally ordered for every prime ideal $P$ of $R$, where $R_P$ denotes the localization at $P$. It is the last condition which limits the generalization of this result to modules or to noncommutative rings. One goal in this manuscript is to generalize the previous result in the noncommutative case. We are interested in which (noncommutative) rings are characterized by the strong algebraic de Morgan's law. Using the techniques of localization in noncommutative Noetherian rings we are able to prove that, under some circumstances, the strong algebraic de Morgan's law characterize Noetherian Asano prime rings (Theorem \ref{asano}). Recall that a prime Goldie ring is said to be \emph{Asano prime ring} if every nonzero ideal is invertible \cite[5.2.6]{mcconnell2001noncommutative}.

A brief description of the present manuscript is as follows:
Section \ref{sec2} is the background material needed to make this manuscript as self-contained as possible. In Section \ref{DML},  we give some results related to  De Morgan's Laws for {\it idiomatic-quantales} 
such that $ab\leq a\wedge b$ holds for all $a, b\in A.$ Main results in this section are Theorem \ref{cuasiani} and Theorem \ref{cuasibaer}, as an
analogous  to \cite[Proposition 4.3.1]{rosenthal1990quantales} given by  P.T. Jonhstone (see also \cite{johnstone1986stone}).  In Section \ref{sec 5}, we study 
the point-free and the point sensitive conditions for modules satisfying the De Morgan's laws (see Theorem \ref{Baer1}), and we achieve a module theoretic counterpart of some results presented in \cite{niefield1995algebraic}. Following this way, in Section \ref{sec 6}, we study properties of $(N:L)$ for fully invariant submodules of a module $M,$ we study the strong algebraic De Morgan's law, and finally, in Theorem \ref{asano} is given a characterization for a Noetherian prime ring  to be an Asano prime ring.
By last,  it is worth to mention that this manuscript is a natural step in our investigation, that was initiated in \cite{mauquantale} followed by \cite{medina2018attaching} and \cite{medina2018strongly}.




\section{Preliminaries.}\label{sec2}
Throughout this paper $R$ will be an associative ring with identity, not necessarily commutative. The word \emph{ideal} will mean two-sided ideal, unless explicitly stated the side (left or right ideal). Module will mean an unital left $R$-module. Let $M$ be an $R$-module, a submodule $N$ of $M$ is denoted by $N\leq M$, whereas we write $N<M$ when $N$ is a proper submodule of $M$. Recall that $N\leq M$ is called \emph{fully invariant submodule}, denoted by $N\leq_{fi}M$, if for every endomorphism $f\in \End_R(M)$, it follows that $f(N)\subseteq N.$ Denote by 
\[\Lambda(M)=\{N\mid N\leq M\},\, \mbox{ and }\,\Lambda^{fi}(M)=\{N\mid N\leq_{fi}M\}.\]
It is well known that both $\Lambda(M)$ and $\Lambda^{fi}(M)$ are complete lattices. Moreover, they are idioms. Recall that  an {\it idiom} is a complete lattice $(A,\leq,\bigvee,\wedge,1 ,0)$, upper-continuous and modular, that is, $A$ is a complete lattice satisfying the following distributive laws: 
\[a\wedge (\bigvee X)=\bigvee\{a\wedge x\mid x\in X\}\leqno({\rm IDL})\]
holds for all $a\in A$  and $X\subseteq A$ directed, and 
\[a\leq b\Rightarrow (a\vee c)\wedge b=a\vee(c\wedge b)\leqno({\rm ML})\]
for all $a,b,c\in A$. 
A good account of the many uses of these lattices can be found in \cite{simmmons1989near} and \cite{simmons2010decomposition}.

A distinguished class of idioms is the class of those which are distributive :
\begin{dfn}
A complete lattice $(A,\leq,\bigvee,\wedge,1,0)$ is a \emph{frame}, if $A$ satisfies
\[a\wedge (\bigvee X)=\bigvee\{a\wedge x\mid x\in X\}\leqno({\rm FDL}),\]
\noindent for all $a\in A$ and $X\subseteq A$ any subset.
\end{dfn}

Of course the main example of a frame comes from topology.  Given a
space $S$ with topology $\mathcal{O}(S)$, it is known that $\mathcal{O}(S)$ is a frame.

The point-free techniques we are interested in are based on the concept of nucleus. We give a quick review of that.
The reader can see  \cite{johnstone1986stone} and \cite{rosenthal1990quantales} for more details of all these facts.

\begin{dfn}\label{nuc}
Let $A$ be an idiom. A \emph{nucleus} on $A$ is a monotone function $j\colon A\rightarrow A$ such that:
\begin{enumerate}
\item $a\leq j(a)$ for each $a\in A$.

\item $j$ is idempotent.

\item $j$ is a \emph{prenucleus}, that is, $j(a\wedge b)=j(a)\wedge j(b)$.
\end{enumerate}
\end{dfn}
\begin{prop}\label{hullker}
Given any morphism of $\bigvee$-semilattices, $f^{*}\colon A\rightarrow B$ there exists  $f_{*}\colon B\rightarrow  A$ such that 
\[f^{*}(a)\leq b\Leftrightarrow a\leq f_{*}(b),\]
  for each $a\in A$ and $ b\in B.$ 
That is, $f^{*}$ and $f_*$ form an adjunction 
{\vspace{-8pt}}
\[\xymatrix{ A\ar@<-.7ex> @/   _.5pc/[r]_--{f^{*}} & B.\ar@<-.7ex>@/ _.5pc/[l]_--{f_{*}} }\]

In fact, $f_{*}(b)=\bigvee \{x\in A\mid f^{*}(x)\leq b\},\,$ for each $b\in B.$
\end{prop}

This is a particular case of the General Adjoint Functor Theorem. A proof of this can be found in any standard book of category theory, for instance, \cite[Theorem 6.3.10]{leinster2014basic}.

We need some other point-free structures that generalize idioms and frames.
As a generalization of quantales, the concept of quasi-quantale was introduced in \cite{mauquantale}.

\begin{dfn}
A \emph{(quasi)quantale} $A$ is a complete lattice with an associative product $A\times A\to A$ such that for all (directed) subsets $X,Y\subseteq A$ and $a\in A:$

\centerline{ $\left(\bigvee X\right)a=\bigvee\{xa\mid x\in X\}$}

and

\centerline{$ a\left(\bigvee Y\right)=\bigvee\{ay\mid y\in Y\}.$}
\end{dfn}


\begin{dfn}\label{sqqcond}
Given a subquasi-quantal $B$ of a quasi-quantale $A$, we will say $B$ satisfies the condition $(\star)$ if $0,1\in B$ and $1b,b1\leq b$ for all $b\in B$.
\end{dfn}

\begin{dfn}[\cite{mauquantale}, Definition 3.16]
Let $B$ be a subquasi-quantale of a quasi-quantale $A$. An element $1\neq p\in A$ is a \emph{prime element relative to $B$} if whenever $ab\leq p$ with $a,b\in B$ then $a\leq p$ or $b\leq p$. We define the \emph{spectrum relative to $B$ of $A$} as
\[Spec_B(A)=\{p\in A\mid p\;\text{is prime relative to }B\}.\]
In the case $A=B$ this is the usual definition of prime element. We denote the set of prime elements of $A$ by $Spec(A)$.
\end{dfn}

\begin{prop}[\cite{mauquantale}]\label{t}
Let $B$ be a subquasi-quantale satisfying $(\star)$ of a quasi-quantale $A$. Then $Spec_B(A)$ is a topological space with closed subsets given by
$ \V(b)=\{p\in Spec_B(A)\mid b\leq p\}$ with $b\in{B}$. 
In dual form, the open subsets are of the form
$\U(b)=\{p\in{Spec_B(A)}\mid b\nleq{p}\}$ with $b\in{B}$.
\end{prop}

\begin{obs}[\cite{mauquantale}]\label{adjunction}
Let $B$ be a subquasi-quantale satisfying $(\star)$ of a quasi-quantale $A$. Let $\mathcal{O}(Spec_B(A))$ be the frame of open subsets of $Spec_B(A)$. We have an adjunction of $\bigvee$-morphisms 
{\vspace{-5pt}}
\[\xymatrix@=20mm{B\ \ \ar@/^/[r]^{\mathcal{U}} & \mathcal{O}(Spec_B(A))\ \ \ar@/^/[l]^{\mathcal{U}_*}}\]
where $\mathcal{U}_*$ is defined as
$\U_*(W)=\bigvee\{b\in{B}\mid\U(b)\subseteq{W}\}.$
The composition $\mu:=\U_*\circ\U$ is a closure operator in $B$. Note that $\U(x)=\U(\mu(x))$ (equivalently, $\V(x)=\V(\mu(x))$) for all $x\in B$.  
\end{obs}








In \cite{BicanPr} is defined a product of modules as follows:

\begin{dfn}\label{pro}
Let $M$ and $K$ be $R$-modules. Let $N\leq M$. The product of $N$ with $K$ is defined as:
\[N_MK=\sum\{f(N)\mid f\in \Hom_R(M,K)\}\]
\end{dfn}

This product generalizes the usual product of an ideal and an $R$-module. For some properties of this product see \cite[Proposition 1.3]{PepeGab}. In particular, we have a product of submodules of a given module. 

\begin{obs}

The condition $(\star)$ of Definition \ref{sqqcond} comes from our canonical example of quasi-quantale $\Lambda(M)$ with $M$ a projective $R$-module in $\sm$ and the canonical subquasi-quantale $\Lambda^{fi}(M)$. Note that $\Lambda^{fi}(M)$ satisfies condition $(\star)$. In general, $\Lambda(M)$ does not satisfies $(\star)$. 
For example, consider $M=\mathbb{Z}_2\oplus\mathbb{Z}_2$, then $\left(\mathbb{Z}_2\oplus 0\right){_MM}=M\nleq\mathbb{Z}_2\oplus 0$.

\end{obs}

\begin{dfn}[Definition 13 \cite{raggiprime}]
Let $N\leq M$ be a proper fully invariant submodule. $N$ is \emph{prime} in $M$ if whenever $L_MK\leq N$ with $L,K\leq_{fi} M$ then $K\leq N$ or $L\leq N$. We say that $M$ is a prime module if $0$ is prime in $M$.
\end{dfn}

\begin{dfn}[\cite{raggisemiprime}]
Let $N\leq M$ be a proper fully invariant submodule. $N$ is \emph{semiprime} in $M$ if whenever $L_ML\leq N$ with $L\leq_{fi} M$ then $L\leq N$. We say that $M$ is a semiprime module if $0$ is semiprime in $M$.
\end{dfn}

Given modules $M$ and $N$, it is said $N$ is \emph{$M$-generated} if there exists an epimorphism $\rho:M^{(X)}\to N$ for some set $X$, where $M^{(X)}$ denotes the direct sum of copies of $M$ indicated by $X$. The module $N$ is said to be \emph{$M$-subgenerated} if $N$ can be embedded in an $M$-generated module. The full subcategory of $R$-Mod consisting of all $M$-subgenerated modules is denoted by $\sm$. The category $\sm$ is a Grothendieck category and many properties can be found in \cite{wisbauerfoundations}. The product of submodules of a module $M$ is neither associative nor distributive over sums from the right, in general. If we assume that $M$ is projective in $\sigma[M]$ then the product is associative and distributive over sums \cite[Proposition 5.6]{beachy2002m}, \cite[Proposition .1.3]{PepeGab} and \cite[Lemma 2.1]{maustructure}. Moreover if $N,L\leq_{fi}M$ then $N_ML\leq_{fi}M$. See \cite[Remark 4.2]{mauquantale}.





\begin{obs}

In the case $A=\Lambda(M)$ and $B=\Lambda^{fi}(M)$, following \cite{medina2018attaching} we write $LgSpec(M)=Spec_{B}(A)$ and we call it \emph{the large spectrum of $M$} and for $Spec_B(B)$ we just write $Spec(M)$. Note that when $R=M$, $Spec(M)$ is the usual prime spectrum. As it was noticed in \cite[Example 4.14]{mauquantale}, if $M$ is quasi-projective then $\mx(M)\subseteq LgSpec(M)$. Moreover, if $M$ is projective in $\sm$, $\emptyset\neq \mx(M)\subseteq LgSpec(M)$ by \cite[22.3]{wisbauerfoundations}.

\end{obs}

\section{The de Morgan's Laws on idiomatic-quantales}\label{DML}

The study of the De Morgan's Laws in (pseudo)multipicative lattices as well as (quasi)quantales, goes back to \cite{ward1939residuated}, and more recently in \cite[Chapter 4]{rosenthal1990quantales}. 

In this section, we give some results related to the  De Morgan's Laws for an {\it idiomatic-quantale} (i.e. an idiom which is also a quantale) such that $ab\leq a\wedge b$ holds for all $a, b\in A$



Let $A$ be a idiomatic-quantale. Since the product of the quantale commutes with arbitrary suprema, we have the adjoint arrow for each variable.

\[(a:\_)\colon A\to A\, \text{ and } \,(\_:a)\colon A\to A\]

given by \[(a:b)=\bigvee\{x\mid ax\leq b\} \mbox{ and } (b:a)=\bigvee\{x\mid xa\leq b\},\] respectively, for all $b\in A$. Note that when $b=0$ then $(a:0)=\ann^{r}(a)$, {\it the right annihilator of $a$ in $A$}, and  $(0:a)=\ann^{l}(a)$ {\it the left annihilator of $a$ in $A$}.
For our purposes in this work, we will consider left annihilators.

Let us say that an idiomatic-quantale satisfies the De Morgan's laws if for all $a, b\in A$:\begin{itemize}

\item[(1)] $\ann(a\vee b)=\ann(a)\wedge\ann(b)$.

\item[(2)] $\ann(a b)=\ann(a)\vee\ann(b)$

The \emph{algebraic De Morgan law} (DML) states that :

\item[(3)]\noindent \[ \ann(a\wedge b)=\ann(a)\vee\ann(b)\]

\end{itemize}

\begin{dfn}
Let $A$ be an idiomatic-quantale (iq for short). We say that $A$ is \emph{semiprime} whenever \[a^{2}=0\Rightarrow a=0,\] for $a\in A$.

\end{dfn}

\begin{lem}\label{ann}
Let $A$ be a semiprime iq then \[\ann(ab)=\ann(a\wedge b)\] holds for all $a,b \in A$.
\end{lem}

\begin{proof}
First, consider  any $a,b\in A,$ then \[[\ann(ab)\ann(ab)]^{2}\leq (\ann(ab)a)(\ann(ab)b).\] Since $a\wedge b\leq a$ then $\ann(ab)(a\wedge b)\leq\ann(ab)a$, and so \[[\ann(ab)(a\wedge b)]^{2}\leq\ann(ab)a(\ann(ab)(a\wedge b))\leq (\ann(ab)a)(\ann(ab)b).\] On the other hand, notice that \[(\ann(ab)a)(\ann(ab)b)=\ann(a)(a\ann(ab)b\leq\ann(ab)ab=0. \] Thus by the semiprime condition, it follows  $\ann(ab)(a\wedge b)=0.$ So, by the supremus property on the (left) annihilator, we get $\ann(ab)\leq \ann(a\wedge b)$. 

Finally, by our hypothesis in this section, $ab\leq a\wedge b.$ Thus, it is always hold that $\ann(a\wedge b)\leq \ann(ab).$ 

\end{proof}

\begin{prop}\label{semipri}
The following conditions are equivalent for   an idiomatic-quantale $A$:

\begin{itemize}
\item[(1)] $A$ is semiprime and satisfies DML.

\item[(2)] $\ann(ab)=\ann(a)\vee\ann(b)$ for all $a,b\in A$.

\item[(3)] $\ann(a)$ is complemented for all $a\in A$ and $A$ satisfies the DML.
\end{itemize}
\end{prop}

\begin{proof} The implication $(1)\implies(2)$ is a direct consequence of Lemma \ref{ann}.

$(2)\implies(3)$ It is enough to observe that \[1=\ann(0)=\ann(\ann(a)a)=\ann(a)\vee\ann(\ann(a)).\]

$(3)\implies(2)$ Now, suppose (3) and consider any $a\in A$ with $a^{2}=0$. Then, by the property of the left annihilator, $a\leq\ann(a).$ Then, by (3) we get that $\ann(\ann(a))\wedge a=0$ and since $a\leq\ann(\ann(a))$ then $a=0$.

\end{proof}

\begin{obs}
Observe that $\ann\ann\colon A\to A$ is monotone idempotent function. 
Moreover, if $A$ satisfies DML, 
then $\ann\ann$ is a multiplicative nucleus. Even more, the quotient $A_{\ann\ann}$ is a complemented subidiom.
\end{obs}

Notice  that  in \cite[Theorem 3.21]{mauquantale}, the authors gave  the construction of a multiplicative nucleus $\mu\colon A\rightarrow A$ such that $A_{\mu}\cong\mathcal{O}(Spec(A))$. Therefore $\mu$ is the point-free version of the radical $\sqrt{\_}$ in commutative ring theory.

\begin{lem}\label{cuasi}
Let $A$ be a semiprime iq then $\ann(a)\in A_{\mu}\,$ for all $a\in A$.
\end{lem}
\begin{proof}
It is enough to prove  $\mu(\ann(a))\leq\ann(a)$.
Indeed,  \[\mu(\ann(a))a\leq \mu(\ann(a))\mu(a)=\mu(\ann(a)a)=\mu(0)=0,\] and notice that the last equality holds by the semiprime condition on $A$. Thus, $\mu(\ann(a))a=0$, and so $\mu(\ann(a))\leq\ann(a)$, as required. 
\end{proof}

In fact, the above result can be done in a more general setting, that is to say, in a multiplicative structure on $A$ (in which the annihilators make sense) with a multiplicative nucleus $j\colon A\rightarrow A$ satisfying that $j(0)=0$.

\begin{lem}\label{cuasi2}
Let $A$ be a iq such that $\mu(0)=0,$ then $\ann(a)=\ann(\mu(a))$ for all $a\in A$.
\end{lem}
\begin{proof}
It will be enough to observe that $\ann(a)\leq\ann(\mu(a))$ for each $a\in A$.
 By the multiplicative condition of $\mu$ and Lemma \ref{cuasi} we get $\mu(\ann(a)a)=\ann(a)\mu(a),$ and since $0=\mu(0)=\mu(\ann(a)a)$, the conclusion is obtained. 
\end{proof}

The next result is analogous  to \cite[Proposition 4.3.1]{rosenthal1990quantales} given by  P.T. Jonhstone. 

\begin{thm}\label{cuasiani} The following conditions are equivalent for a semiprime quantale 
$A$.
\begin{enumerate}[\rm(1)]
\item $A$ satisfies DML.
\item $\mathcal{O}(Spec(A))$ satisfies DML.
\item $Spec(A)$ is extremely disconnected.
\end{enumerate}
\end{thm}
\begin{proof}

Suppose (1). Consider any $a,b\in A_{\mu},$ then by Lemma \ref{cuasi} we get $\ann(a),\ann(b)\in A_{\mu}.$ Thus, $\neg a=\ann(a)$ (the negation in the frame $A_{\mu}$), and so, we conclude (2).

Now if we suppose that in the frame $A_{\mu}$ DML holds, consider any $a,b\in A$ then again by Lemma \ref{cuasi} we get $\ann(a), \ann(b)\in A_{\mu} $ then \[\ann(a\wedge b)=ann(\mu(a\wedge b))=\ann(\mu(a)\wedge\mu(b))=\neg(\mu(a)\wedge\mu(b))=\neg\mu(a)\vee\neg\mu(b)\]\[=\ann(\mu(a))\vee\ann(\mu(b))=\ann(a)\vee\ann(b).\] as required.

The equivalence (2)$\Leftrightarrow$ (3) it follows by 
\cite[Theorem 1]{johnstone1979conditions}.

\end{proof}

\begin{thm}\label{cuasibaer}
The following conditions are equivalent for a semiprime iq $A$.\begin{itemize}
\item[(1)] $A$ satisfies DLM.

\item[(2)] $\ann(ab)=\ann(a)\vee\ann(b)$ for all $a,b\in A$.

\item[(3)] $\ann(a)$ is complemented for all $a\in A$ and $A$ satisfies DML.

\item[(4)]
$\mathcal{O}(Spec(A))$ satisfies DML.
\item[(5)] $Spec(A)$ is extremely disconected.

\end{itemize}
\end{thm}
\begin{proof}
It is a consequence of Proposition \ref{semipri} and Theorem \ref{cuasiani}.
\end{proof}

In \cite{simmons1989compact} the author introduces a frame to study the \emph{rather below relation} $\curlyeqprec$ for a class of quantales named \emph{carries}. These are just two sided quantales, which are also upper continuous lattices and the quantale product is compatible with the order.
Note that every idiomatic two sided quantale is a carrier. The idea of studying these structures is to give a general overview of the regularity for a frame $\Omega$. Recall that for any $a,x\in\Omega$ we say that $x$ is \emph{rather below} $a$, denoted by $x\curlyeqprec a$, if $a\vee\neg x=1$. Let $\Omega^{\neg}$ denote the set of all elements of $\Omega$ such that $a=\bigvee\{x\mid x\curlyeqprec a\}$. It is said that a frame $\Omega$ is \emph{regular} if $\Omega=\Omega^{\neg}$ (see \cite{johnstone1986stone} for details).

Thus as it is shown in \cite{simmons1989compact}, the rather below relation can be stated in terms of the product: 

Consider the operator $r\colon A\rightarrow A$ with $A$ a quantale given by \[r(a)=\bigvee\{x\in A\mid x \curlyeqprec a\}\] where $a\curlyeqprec b$ if and only if $\ann(a)\vee b=1$, see \cite[page 501]{simmons1989compact} for more details.
Let us denote  \[\Psi(A):=\{a\in A\mid r(a)=a\}.\] As in \cite{simmons1989compact}, this set is a frame and in general is not regular, but $\Psi(A)$ has its own $r$, therefore we can extract a regular part  of $A$ called \emph{the regular core} by iterating the above process:
Let $A^{r(2)}:=\Psi(A)^{r}$. Inductively, it is defined:
\[A^{r(0)}:=A\;\; A^{r(\alpha+1)}:= A^{r(\alpha)r}\;\; A^{r(\lambda)}:=\bigcap\{A^{r(\alpha)}\mid\alpha<\lambda\}\] 
for each non-limit ordinal $\alpha$ and limit ordinal $\lambda$ respectively. This chain is decreasing, therefore by a cardinality argument it eventually stabilizes in some ordinal. Let us denote the least of those ordinals by $\infty$ and $A^{reg}:=A^{r(\infty)}$. In \cite[Theorem 3.4]{simmons1989compact} it is proved that $A^{reg}$ is a regular frame and every regular subframe of $A$ is contained in it. 

Next, we will see the implications of Proposition \ref{semipri} in connection with the regularity of the frame $\Psi(A)$.

\begin{lem}\label{need}
Let $A$ be a iq such that $\ann(a)$ is complemented for each $a\in A$. Then $r(\ann(a))=\ann(a)$ for every $a\in A$.  
\end{lem}

\begin{proof}
First, by the hypothesis,  we get $\ann(a)\leq r(\ann(a))$.
On the other hand, notice that $r$ is deflatory, so in particular, $r(\ann(a))\leq \ann(a) $ 
\end{proof}

\begin{lem}\label{corrr}
In a semiprime iq $A$, if $ab=0$ then $ba=0$ and $a\wedge b=0$.
\end{lem}

\begin{proof}
Indeed, $(ba)^{2}=(ba)(ba)=b(ab)a=b(0)a=0$ then $ba=0$.

For the other requirement note that $(a\wedge b)^{2}=(a\wedge b)(a\wedge b)\leq ab=0$ thus the semiprime condition gives the result.
\end{proof}


\begin{prop}\label{cori}
Let $A$ be a semiprime iq satisfying DML then the frame $\Psi(A)$ satisfies DML. Moreover, $\Psi(A)$ is a 
regular frame.  
\end{prop}

\begin{proof}
By Theorem \ref{cuasibaer}, $A$ satisfies DML then $r(\ann(a))=\ann(a)$ for every $a\in A.$ The semiprime condition on $A$ ensures by Lemma \ref{corrr} $\ann(a)\wedge a=0.$ Thus, if $a\in\Psi(A)$ we have $\ann(a)=\neg a.$ Therefore, $\Psi(A)$ satisfies DML.

Observe that in this situation, if $a\vee\ann(x)=1$ and $a\in\Psi(A)$, then \[x\ann(x)\leq x\wedge\ann(x)=\ann(x)x=0,\] 
therefore, $x\leq\ann\ann(x).$ Thus \[a\leq \bigvee\{\ann\ann(x)\in\Psi(A)\mid a\vee\ann\ann\ann(x)=1\}\leq r(a),\] and hence $r^{2}=r$. Thus the regular core is just $\Psi(A)$.

\end{proof}

Given a complete lattice $\ltc,$ recall that an element $c\in \ltc$ is said to be \textit{compact} if  for every $X\subseteq \ltc$ such that $c\leq \bigvee X,$ there exists a finite subset $F\subseteq X$ satisfying $c\leq \bigvee F.$
Also, recall that a lattice $\ltc$ is said to be a \emph{compact lattice} if and only if $1_\ltc$ is compact in $\ltc$. Recall that an idiomatic quantale $A$ is  said to be \emph{normal} if for every $a,b\in A$ with $a\vee b=1,$ there exist $a',b'\in A$ such that $a\vee b'=1=a'\vee b$ and $a'b'=0,$ see also  
\cite[Definition 3.1]{medina2018strongly}.

\begin{cor}\label{nor}
Let $A$ be a compact normal iq. If $A$ is semiprime and DML holds then $\pt\Psi(A)$ is a extremely disconnected Hausdorff space.
\end{cor}

\begin{proof}
By proposition \ref{cori} and \cite[Theorem 3.5]{simmons1989compact} we have $\pt\Psi(A)\cong Max(A)$ which is Hausdorff.
\end{proof}

In the next section we will see the module theoretic counterpart of Theorem \ref{cuasibaer}, Proposition \ref{cori} and Corollary \ref{nor}.


\section{De Morgan's Laws  }\label{sec 5}

The following are the module theoretic counterpart of some results given in \cite{niefield1995algebraic}. First, we shall provide some properties of the annihilator $Ann_M(-)$ on modules, using the results in the previous section, which will be useful for the following sections.

\begin{dfn}\cite{beachy2002m} Let $M$ and $K$ be $R$-modules. The \emph{annihilator of $K$ in $M$} is defined as:

\centerline{$\Ann_M(K)=\bigcap\{Ker(f)\mid f\in\Hom_{R}(M,K)\}.$}
\end{dfn}

This annihilator is a fully invariant submodule of $M$ and it is the greatest submodule of $M$ such that $\Ann_M(K)_MK=0$. 






\begin{obs}
If $M$ is a module, the lattice of (resp. fully invariant) submodules $\Lambda(M)$ (resp. $\Lambda^{fi}(M)$) is an idiom. Moreover, if $M$ is projective in $\sm$, this idiom is an idiomatic-quantale. Therefore, the results in Section \ref{DML} can be applied here.
\end{obs}

\begin{lem}\label{semiprimelemma}
Let $M$ be a semiprime module projective in $\sm$ and let $N,L\leq M$. If $N_ML=0$ then $L_MN=0$ and $N\cap L=0$.
\end{lem}

\begin{proof}
It follows from Lemma \ref{corrr}. (See also \cite[Lemma 1.9]{maugoldie})
\end{proof}


\begin{lem}\label{annprodinter}
	Let $M$ be a semiprime module projective in $\sm$ and let $N,L\in\Lambda^{fi}(M)$. Then $\Ann_{M}(N\cap L)=\Ann_{M}(N_{M}L)$.
\end{lem}

\begin{proof}
It follows from Lemma \ref{ann}.
%
\end{proof}

\begin{lem}\label{ann1}
Let $M$ be a semiprime module projective in $\sm$. Then, for all $N, L, N'\text{ and } L'\in\Lambda^{fi}(M)$ such that $\Ann_M(N)=\Ann_M(N')$ and $\Ann_M(L)=\Ann_M(L')$, it is satisfied that
 \[\Ann_{M}(N_ML)=\Ann_{M}({N'}_ML')\] 
\end{lem}

\begin{proof}
We have that $\Ann_{M}(N_{M}L)_M N_{M}L=0$. Since $\Ann_M(L)=\Ann_M(L')$, $(\Ann_{M}(N_{M}L)_M N)_{M}L'=0$. It follows from Lemma \ref{semiprimelemma} that \[N_M(L'_M \Ann_{M}(N_{M}L))=0.\] 
Since $\Ann_M(N)=\Ann_M(N')$, after two applications of Lemma \ref{semiprimelemma} we get $N'_M(L'_M \Ann_{M}(N_{M}L))=0$. Again Lemma \ref{semiprimelemma} implies that $\Ann_{M}(N_{M}L)\subseteq \Ann_{M}(N'_{M}L')$. The other comparison is similar.
\end{proof}

In the previous section was introduced the frame $\Psi(A)$ for an idiomatic-quantale $A$. The module-theoretic counterpart of this frame was presented and studied in \cite[Section 5]{medina2018attaching}. Given a module $M$, we defined the following spatial frame \[\Psi(M)=\{N\in\Lambda^{fi}(M)\mid \forall n\in N, N+\Ann_M(Rn)=M\}.\]

If $M$ is a self-progenerator in $\sigma[M]$, the frame $\Psi(M)$  is characterized as the fixed points of an operator named $Ler\colon \Lambda^{fi}(M)\to  \Lambda^{fi}(M)$ given by $Ler(N)=\{m\in M \mid N+\Ann_M(Rm)=M \}$ and this operator coincides with the operator $r$ of last section that defines $\Psi(A)$ \cite[Proposition 5.19]{medina2018attaching}.

Properties of this operator are given in \cite{medina2018attaching}. Also, in \cite{medina2018strongly}, more characterizations of this operator and the frame $\Psi(M)$ are obtained when $M$ is a strongly harmonic or Gelfand module.

\begin{dfn}\label{Baer}
A module $M$ is called \emph{FI-Baer} if $\Ann_{M}(N)$ is a direct summand of $M$ for every $N\in\Lambda^{fi}(M)$
\end{dfn}

\begin{obs}
In \cite[Definition 3.2]{rizvibaer} are defined \emph{quasi-Baer modules}. Given $N\in\Lambda^{fi}(M)$, $\Hom_R(M,N)$ is a two-sided ideal of $\End_R(M)$. Then by \cite[Remark 3.3]{rizvibaer} in every quasi-Baer module $M$, $\Ann_M(N)$ is a direct summand for all $N\in\Lambda^{fi}(M)$.
\end{obs}

The following example shows that there is an FI-Baer module which is not quasi-Baer.

\begin{ej} Let $p\in\mathbb{Z}$ be a prime number and consider the $\mathbb{Z}$-module $M=\mathbb{Z}_{p^\infty}$. Given any proper submodule $N$ of $M$, $\Hom_R(M,N)=0$. This implies that $\Ann_M(N)=M$. It follows that $M$ is an FI-Baer module. On the other hand, consider the submodule $\mathbb{Z}_p\leq M$ and let $S=\End_\mathbb{Z}(M)$. Recall that $S$ is an integral domain, then 
$0\neq\ell_S(\mathbb{Z}_p)=\{f\in S\mid f(\mathbb{Z}_p)=0\}<S.$
Since $S$ is a domain, $\ell_S(\mathbb{Z}_p)$ cannot be generated by an idempotent. Therefore, $M$ is not quasi-Baer.
\end{ej}

\begin{prop}\label{Semiler}
Let $M$ be a semiprime module self-projective in $\sm$. If $M$ is $FI$-Baer, then $Ler(\Ann_{M}(N))=\Ann_{M}(N)$ for all $N\in\Lambda^{fi}(M)$.
\end{prop}

\begin{proof}
It follows from Lemma \ref{need} because in this case, $r=Ler$.
%
\end{proof}

\begin{cor}
Let $M$ be a semiprime module self-projective in $\sm$. If $M$ is $FI$-Baer, then $\Psi(M)=(\Lambda^{fi}(M))^{reg}$.
\end{cor}

\begin{proof}
It follows from Proposition \ref{Semiler} and \cite[Theorem 5.20]{medina2018attaching}.
\end{proof}

\begin{prop}\label{semi}
Let $M$ be a semiprime module self-projective in $\sm$. Suppose that $M=N\oplus L$. If $N\in\Lambda^{fi}(M)$ then $L\in\Lambda^{fi}(M)$
\end{prop}

\begin{proof}
We have that $N_{M}L\subseteq N\cap L=0$, therefore $N_{M}L=0$.
Now $N_{M}L=0=L_{M}N$, since $M$ is semiprime. Thus $L_{M}M=\left( L_{M}N\right) \oplus L^{2}=L^{2}\leq L$. It follows that $L=L_{M}M$.
\end{proof}

\begin{lem}\label{idem}
If $M=N\oplus L$ with $N, L\in\Lambda^{fi}(M)$ then the corresponding idempotent $e_{N}\colon M\rightarrow N$ is central.
\end{lem}

\begin{proof}
Let $f\in \End_{R}(M)$ then $fe_{N}(m)=fe_{N}(n+l)=f(n)$ and $e_{N}f(m)=e_{N}f(n)+e_{N}f(l)=f(n)$.
\end{proof}

Recall that a module $M$ is \emph{FI-retractable} if for any non-zero fully invariant submodule $K$ of $M,$ $\Hom_R(M, K) \neq 0.$
Conditions on FI-retractable module has been study previously in the literature, for instance see
\cite{mostafanasab2015endoprime}.
 It can be seen that if $M$ is semiprime, then $M$ is $FI$-retractable, following a similar proof to that given in \cite[Lemma 1.24]{maugoldie}.

\begin{dfn} 
A module $M$ satisfies the \emph{ algebraic de Morgan's law} (DML) if \[\Ann_{M}(N\cap L)=\Ann_{M}(N)+\Ann_{M}(L).\] for all $N, L\in\Lambda^{fi}(M).$
\end{dfn}

\begin{ej}
Consider the $\mathbb{Z}$-module $M=\mathbb{Z}_{p^\infty}$. We claim that $M$ satisfies DML. Consider $N$ and $L$ two submodules of $M$. We have the following cases:
\begin{enumerate}
\item $N$ and $L$ are proper submodules. Then $N\cap L$ is proper in $M$. Hence $\Ann_M(N)=M=\Ann_M(L)$ and $\Ann_M(N\cap L)=M$.
\item $N=M$ and $L$ is proper. Then $N\cap L=L$ and so $\Ann_M(N\cap L)=\Ann_M(L)=M$. On the other hand, $\Ann_M(N)=0$ and $\Ann_M(L)=M$.
\item $N=M=L$. In this case, $\Ann_M(N\cap L)=0$ and $\Ann_M(N)=0=\Ann_M(L)$.
\end{enumerate}
Therefore, $M$ satisfies DML.
\end{ej}



\begin{thm}\label{Baer1}
The following conditions are equivalent for a module $M$ projective in $\sm$:
\begin{enumerate}[\rm(1)]
\item $M$ is semiprime and satisfies DML.

\item $M$ is FI-retractable and $\Ann_{M}(N_{M}L)=\Ann_{M}(N)+\Ann_{M}(L)$ for all $N,L\in\Lambda^{fi}(M)$.

\item $M$ is FI-retractable and $M=\Ann_{M}(N)\oplus \Ann_{M}(\Ann_{M}(N))$ for all $N\in\Lambda^{fi}(M)$.

\item $M$ is a FI-retractable FI-Baer module and for all $N\in\Lambda^{fi}(M)$ the projection \[\pi\colon M\rightarrow \Ann_{M}(N)\] is a central idempotent of $\End_R(M)$.

\item $M$ is semiprime and $SP(M)$ satisfies DML.

\item $M$ is semiprime and $Spec(M)$ is extremely disconnected.
\end{enumerate}

\end{thm}

\begin{proof}
(1)$\Rightarrow$(2) Since $M$ semiprime, $M$ is FI-retractable. It follows from Lemma \ref{annprodinter} that
\[\Ann_{M}(N\cap L)=\Ann_{M}(N_{M}L)\] for all $L, N\in\Lambda^{fi}(M)$. Thus, by hypothesis
\[\Ann_{M}(N_{M}L)=\Ann_{M}(N)+\Ann_{M}(L).\]

(2)$\Rightarrow$(3) Let us first see that (2) implies that $M$ is a semiprime module. Consider $N\in\Lambda^{fi}(M)$ such that $N_{M}N=0$ then $M=\Ann_{M}(N_{M}N)=\Ann_{M}(N)+\Ann_{M}(N)=\Ann_{M}(N)$. Since $M$ is FI-retractable, $N=0$, proving that $M$ is semiprime. It follows that 
\[M=\Ann_{M}(\Ann_{M}(N)_{M}N)=\Ann_{M}(\Ann_{M}(N))+\Ann_{M}(N)\]
by (2). This decomposition is direct by Lemma \ref{semiprimelemma}.

(3)$\Rightarrow$(4) The Baer property is immediate and the central idempotent requirement follows from Lemma \ref{idem}.

(4)$\Rightarrow$(1) Let $K$ be a fully invariant submodule of $M$ such that $K_{M}K=0$. By hypothesis, $M=\Ann_{M}(K)\oplus L$. Let $\pi$ be the projection onto $\Ann_{M}(K)$, then $\pi(K)=K$. Consider any morphism $f\colon M\rightarrow K$. By hypothesis, $\pi f=f\pi$. Then 
\[f(M)=\pi f(M)= f\pi(M)=f(\Ann_M(K))=0.\]
Therefore $K=0$ because $M$ is FI-retractable. Thus $M$ is a semiprime module.

It only remains to prove that DML holds. Let $N, L\in\Lambda^{fi}(M)$ then \[M=\Ann_{M}(N)\oplus N'\;\text{ and }\; M=\Ann_{M}(L)\oplus L'.\]

We claim that $\Ann_{M}(N)=\Ann_{M}(N')$ and $\Ann_{M}(L)=\Ann_{M}(L')$. By symmetry, it is enough to prove just one of these equalities.

Observe that $\Ann_{M}(N)_MN'\subseteq \Ann_{M}(N)\cap N'=0$ thus $\Ann_{M}(N)\subseteq \Ann_{M}(N')$. Now if $K$ is a submodule such that $K_MN'=0$ then \[K\subseteq K_{M}M= \left[ K_{M}\Ann_{M}(N)\right] \oplus \left[ K_{M}N'\right] =K_M \Ann_{M}(N)\subseteq \Ann_{M}(N).\] 
Similarly, $\Ann_M(L)=\Ann_M(L')$.
Note that  $\Ann_{M}(N'_{M}L')_{M}N'\subseteq \Ann_{M}(L')$. 
Now consider $x\in \Ann_{M}(N'_ML')$ and let $e_{N'}$ denote the canonical projection onto $N'$. Then $x=e_{N'}(x)+(1-e_{N'})(x)$. It follows $(1-e_{N'})(x)\in \Ann_M(N)=\Ann_{M}(N')$. Hence $\Ann_{M}(N'_{M}L')\subseteq \Ann_{M}(L')+\Ann_{M}(N')$ and so, the equality holds. By Lemma \ref{ann1}, \[\Ann_{M}(N_ML)=\Ann_{M}(N'_ML')=\Ann_{M}(N')+\Ann_{M}(L').\] 
It follows from Lemma \ref{annprodinter} that $M$ satisfies DML.

(1)$\Rightarrow $(5),(6) 
Let $N$ be a semiprime submodule of $M$. By \cite{maugoldie}, $\Ann_M(N)$ is in $SP(M)$. Note that if $N,L\in SP(M)$, then $N_ML=N\cap L$. Since $\Ann_M(N)_MN=0$, it follows that $\neg N=\Ann_M(N)\in SP(M)$. Since $M$ satisfies DML, so does $SP(M)$. We have that $SP(M)\cong\mathcal{O}(Spec(M))$ as frames. Hence $Spec(M)$ is extremely disconnected by	 \cite{johnstone1979conditions} and \cite[Proposition 4.3.1]{rosenthal1990quantales}.

(5)$\Rightarrow $(6) This equivalence is just an application of \cite[Theorem 1]{johnstone1979conditions}, see also \cite[Proposition 4.3.1]{rosenthal1990quantales}

(5)$\Rightarrow $(1) 
Let $N,L\in\Lambda^{fi}(M).$  
Since $SP(M)$ satisfies DML. Let $L,N\in\Lambda^{fi}(M)$
Then, $\Ann_M(L\cap N)=\Ann_M(L_MN)$ by \ref{annprodinter}

First,
notice that \[0=\mu(0)=\mu(\Ann_M(N)_MN)=\mu(\Ann_M(N))_M\mu(N)\geq \Ann_M(N)_M \mu(N)\]
where $\mu$ is the nucleus given by Remark \ref{adjunction}.
Then, $\Ann_M(N)\leq \Ann_M(\mu(N)). $
It always holds $\Ann_M(\mu(N))\leq \Ann_M(N),$ then we conclude $\Ann_M(\mu(N))=\Ann_M(N)$ for each $N\in \Lambda^{fi}(M).$

Also, notice that for each $L\in SP(M)=\Lambda^{fi}(M)_{\mu},$ we have $\neg L=\Ann_M(L).$ 

Then, applying (5) and the fact that $Ann(L\cap N)=Ann(L_MN)$ by Lemma \ref{annprodinter}, we get that for all $L,N\in\Lambda^{fi}(M),$

\[\Ann_M(L\cap N)=\Ann_M(L_MN)=\Ann_M(\mu(L_MN))=\Ann_M(\mu(L)\wedge \mu(N)) \]\[=\Ann_M(\mu(L))+\Ann_M(\mu(N))=\Ann_M(L)+\Ann_M(N).\]

This is, $\Ann_M(L\cap N)=\Ann_M(L)+\Ann_M(N).$ Therefore, $M$ satisfies DML.

\end{proof}




\begin{cor}\label{psidml}
	Let $M$ be self-progenerator in $\sm$. If $M$ is semiprime and satisfies DML, then the frame $\Psi(M)$ satisfies DML. Moreover $\Psi(M)$ is a regular frame.
\end{cor}

\begin{proof}
	By Theorem \ref{Baer1}, $M$ is FI-Baer. It follows from Proposition \ref{Semiler} that $Ler(\Ann_M(N))=\Ann_M(N)$ for all $N\in\Lambda^{fi}(M)$, that is, $\Ann_M(N)$ is in $\Psi(M)$ for all $N\in\Lambda^{fi}(M)$. Since $M$ is semiprime $\Ann_M(N)_MN=0$ implies that $\Ann_M(N)\cap N=0$. Hence, if $N\in\Psi(M)$ then $\neg N=\Ann_M(N)$ is in $\Psi(M)$. Thus $\Psi(M)$ satisfies DML. It follows from \cite[Corollary 5.22]{medina2018attaching} 
that $\Psi(M)$ is a regular frame.
	
\end{proof}

For the definition of \emph{Strongly harmonic } module see \cite{medina2018strongly}.

\begin{cor}
	Let $M$ be a self-progenerator in $\sm$. Suppose $M$ is strongly harmonic such that $\Lambda^{fi}(M)$ is compact. If $M$ is semiprime and satisfies DML then $Max^{fi}(M)$ is extremely disconnected.
\end{cor}

\begin{proof}
	By \cite[Corollary 4.27]{medina2018strongly}, $\Psi(M)\cong\mathcal{O}(Max^{fi}(M))$. It follows from Corollary \ref{psidml} that $\mathcal{O}(Max^{fi}(M))$ satisfies DML. Thus $Max^{fi}(M)$ is extremely disconnected by \cite{johnstone1979conditions}.
\end{proof}

Following \cite[pp. 130]{niefield1995algebraic}, a ring $R$ is Baer if $R$ is semiprime and $\Ann(I)$ is generated by a central idempotent for every ideal $I$ of $R$.

\begin{cor}
	Let $R$ be a Baer ring. If $R$ is strongly harmonic, then $Max(R)$ is extremely disconnected.
\end{cor}

\section{Strong Algebraic De Morgan's Law}\label{sec 6}

In \cite{niefield1985strong}, the authors show that the strong algebraic de Morgan's law characterize commutative Dedekind domains, as follows:

\begin{thm}\cite[Theorem 2.4]{niefield1985strong}
The following are equivalent for a Noetherian domain $R$.
\begin{enumerate}
\item[(1)] $R$ is a Dedekind domain;
\item[(2)] $(A:B)+(B:A)=R$, for all $A,B\in \Lambda(R)$;
\item[(3)] $(A+B):C=(A:C)+(B:C)$, for all $A,B,C\in \Lambda(R)$;
\item[(4)] $A:(B\cap C)=(A:B)+(A:C)$, for all $A,B,C\in \Lambda(R)$.
\item[(5)] $\Lambda(R_P)$ is totally ordered, for every prime ideal $P$ of $R$, where $R_P$ denotes the localization of $R$ at $P$.
\item[(6)] $A\cap(B+C)=(A\cap B)+(A\cap C)$, for all $A,B,C\in\Lambda(R)$;
\item[(7)] $A(B\cap C)=AB\cap AC$, for all $A,B,C\in\Lambda(R)$;
\item[(8)] $A+(B\cap C)=(A+ B)\cap(A+ C)$, for all $A,B,C\in\Lambda(R)$.
\end{enumerate}
\end{thm}

Our aim is to generalize this theorem to the non-commutative case. We want to know which of those eight conditions remain equivalent for noncommutative rings and which kind of rings are characterized by them. So, let $M$ be an arbitrary $R$-module over an associative ring $R$. Given two fully invariant submodules $N$ and $L$ of $M$, let $(N:L)$ denote the following subset of $M$:
\[(N:L)=\{m\in M\mid f(m)\in N,\;\forall f\in\Hom_R(M,L)\}.\]

Some basic properties of $(N:L)$ can be found in \cite[Lemma 1.6]{beachy2019fully}, and \cite{MR3912148}. In particular, $(N:L)$ is a fully invariant submodule of $M$, and it is easy to see that $(N:\_)$ is the right adjoint of $N_M\_$ provided that $M$ is projective in $\sm$. 

Note that if $I$ and $J$ are ideals of a ring $R$, then 
\[(I:_\ell J)=\{r\in R\mid f(r)\in I,\;\forall f\in\Hom_R(M,J)\}=\{r\in R\mid rJ\subseteq I\}.\]

\begin{prop}\label{dospuntos}
Let $M$ be a module and $N,L,K\in\Lambda^{fi}(M)$. Then,
\begin{enumerate}
\item $N\geq L$ $\Rightarrow$ $(N:L)=M$.
\item $N\geq L$ $\Rightarrow$ $(N:L)\geq(L:K)$ and $(K:L)\geq(K:N)$.
\item $(N\cap L:K)=(N:K)\cap (L:K)$.
\item $(N:L+K)\subseteq (N:L)\cap (N:K)$.
\item $(N:L)=M$ $\Rightarrow$ $N\geq L$, provided that $M$ generates all its fully invariant submodules.
\item $(N:L)\cap (N:K)\subseteq (N:L+K)$, provided that $M$ is quasi-projective.
\end{enumerate}
\end{prop}

\begin{proof}
\textit{(1)} Since $L\leq N$, $f(M)\leq L\leq N$ for all $f\in\Hom_R(M,L)$. Thus, $M=(N:L)$.

\textit{(2)} Suppose $L\leq N$. Let $x\in (L:K)$ and $f:M\to K$. Then $f(x)\in L\leq N$. This implies that $x\in (N:K)$. Thus, $(L:K)\leq (N:K)$. Now, let $x\in(K:N)$, i.e., $f(x)\in K$ for all $f\in\Hom_R(M,N)$. Let $g\colon M\to L$ be any homomorphism. Since $L\leq N$, $g$ can be seen as a homomorphism $g:M\to N$. Then $g(x)\in K$. Thus, $x\in (K:L)$ and so, $(K:N)\leq (K:L)$.

\textit{(3)} By definition 
\[(N\cap L:K)=\{m\in M\mid f(m)\in N\cap L\;\forall f\in\Hom_R(M,K)\}.\]
Hence, $(N\cap L:K)\leq (N:K)\cap (L:K)$. On the other hand, let $x\in (N:K)\cap (L:K)$. Then, for each $f\in\Hom_R(M,K)$ we have that $f(x)\in N$ and $f(x)\in L$. Thus, $x\in (N\cap L:K)$.

\textit{(4)} Let $m\in (N:K+L)$. Since $\Hom_R(M,L)\subseteq \Hom_R(M,K+L)$, $f(m)\in N$ for all $f\in\Hom_R(M,L)$. Analogously, $g(m)\in M$ for all $g\in\Hom_R(M,K)$. Then $m\in (N:L)\cap(N:K)$.

\textit{(5)} If $(N:L)=M$, then $M_ML\leq N$. By hypothesis, $L=M_ML$. Thus, $L\leq N$.

\textit{(6)} By \cite[18.4]{wisbauerfoundations}, $\Hom_R(M,L+K)=\Hom_R(M,L)+\Hom_R(M,K)$. Let $x\in (N:L)\cap (N:K)$ and $f\in\Hom_R(M,L+K)$. Then, there exist $g\in\Hom_R(M,L)$ and $h\in\Hom_R(M,K)$ such that $f=g+h$. Then $f(x)=g(x)+h(x)\in N$. Thus, $x\in (N:L+K)$.
\end{proof}

\begin{prop}\label{sdmls}
Consider the following conditions for fully invariant submodules $N,L,K$ of $M$:
\begin{itemize}
\item[SDML] $(N:L)+(L:N)=M$.
\item[SDML1] $(N+L):K=(N:K)+(L:K)$.
\item[SDML2] $N:(L\cap K)=(N:L)+(N:K)$.
\end{itemize}
Then,
\begin{enumerate}
\item \emph{SDML2} $\Rightarrow$ \emph{SDML}.
\item If $M$ quasi-projective and generates all its fully invariant submodules, then \emph{SDML2} implies that $\Lambda^{fi}(M)$ is distributive.
\item If $\Lambda^{fi}(M)$ is a quantale, then the three conditions are equivalent.
\end{enumerate}
\end{prop}

\begin{proof}
\textit{(1)} Let $N,L\in\Lambda^{fi}(M)$ and suppose SDML2. By Proposition \ref{dospuntos}\textit{(3)} and \textit{(1)}, it follows that:
\[(N\cap L):N=(N:N)\cap (L:N)=M\cap (L:N)=(L:N),\]
and
\[(N\cap L):L=(N:L)\cap (L:L)=(N:L)\cap M=(N:L).\]
Therefore, 
\[(N:L)+(L:N)=((N\cap L):L)+((N\cap L):N)=(N\cap L):(L\cap N)=M.\]

\textit{(2)} Let $N,L,K\in\Lambda^{fi}(M)$ and suppose $L\cap K\subseteq N$. By Proposition \ref{dospuntos} and SDML2,
\begin{equation*}
\begin{split}
(N:(N+L)\cap(N+K)) & =(N:N+L)+(N:N+K) \\
& \subseteq ((N:N)\cap(N:L))+((N:N)+(N:K)) \\ 
& =(N:L)+(N:K) \\ 
& =N:(L\cap K) \\
& =M.
\end{split}
\end{equation*}
Since $M$ is quasi-projective and generates all its fully invariant submodules, by Proposition \ref{dospuntos}\textit{(5)} and \textit{(6)}, we have that $(N:(N+L)\cap(N+K))=M$. Therefore, $(N+L)\cap(N+K)\subseteq N$. Since always $N\subseteq (N+L)\cap(N+K)$, $N=(N+L)\cap(N+K)$. Now, let $N,L,K\in\Lambda^{fi}(M)$ any fully invariant submodules. Then $L\cap K\subseteq N+(L\cap K)$. By the above, 
\[N+(L\cap K)=(L+N+(L\cap K))\cap(K+N+(L\cap K))=(N+L)\cap(N+K).\]

\textit{(3)} It follows from \cite[Proposition 3.2]{niefield1995algebraic} and the paragraph below that proposition. 
\end{proof}

Suppose that $S$ is a multiplicatively closed subset of $R$ consisting of regular elements such that $S^{-1}R$ exists. Recall that the \emph{extension} of an ideal $I$ of $R$ is given by the subset $I^e=\{s^{-1}a\mid s\in S\;a\in I\}\subseteq S^{-1}R$. If $R$ is Noetherian then $I^e$ is an ideal of $S^{-1}R$ \cite[Theorem 10.18]{goodearl2004introduction}.

\begin{lem}\label{extdospuntos}
Let $I$ and $J$ be ideals of the Noetherian ring $R$ and let $S$ be a multiplicatively closed subset of $R$ consisting of regular elements such that $S^{-1}R$ exists. Then $(I^e:J^e)\subseteq (J:I)^e$.
\end{lem}

\begin{proof}
Since $R$ is Noetherian, we can assume $J=b_1R+\cdots+b_nR$. Let $s^{-1}a\in (I^e:J^e)$, then $s^{-1}ab_1\in I^e$. Hence, there exists $c\in I$ and $t\in S$ such that $s^{-1}ab_1=t^{-1}c$. This implies that there exist $x_1,y_1\in R$ such that $x_1ab_1=y_1c\in I$ and $x_1s=y_1t\in S$ \cite[Lemma 6.1]{goodearl2004introduction}. Note that $(x_1s)^{-1}x_1a=s^{-1}a\in (I^e:J^e)$. Then, for $b_2$ there exists $x_2\in R$ such that $x_2x_1ab_2\in I$ and $x_2x_1s\in S$. Moreover $(x_2x_1s)^{-1}x_2x_1a=s^{-1}a\in (I^e:J^e)$. Doing this for each $b_i$, there exist $x_1,\dots,x_n\in R$ and such that $x_i\cdots x_1ab_i\in I$ and $x_i\cdots x_1s\in S$ for all $1\leq i\leq n$. Thus, $(x_n\cdots x_1a)(b_1+\cdots +b_n)\in I$. This implies that $x_n\cdots x_1a\in (I:J)$, and so $s^{-1}a=(x_n\cdots x_1s)^{-1}x_n\cdots x_1a\in (I:J)^e$.
\end{proof}

\begin{prop}\label{idealesiguales}
Let $R$ be a Noetherian ring such that each maximal ideal of $R$ is localizable, and let $I$ and $J$ be ideals of $R$. If $IR_P=JR_P$ for each maximal ideal $P$, then $I=J$.
\end{prop}

\begin{proof}
Let $P$ be a maximal ideal of $R$ and let $\lambda:R\to R_P$ denote the canonical homomorphism. If $I^e=IR_P=JR_P=J^e$ then for each $a\in I$ there exist $b\in J$ and $s,u\in\mathcal{C}_R(P)$ such that $\lambda(u)^{-1}\lambda(b)=\lambda(s)^{-1}\lambda(a)$. It follows form \cite[Lemma 10.1]{goodearl2004introduction} that there exist $r\in R$ and $t\in\mathcal{C}_R(P)$ such that $rb=ta$ and $ru=ts$. Since $b\in J$, then $ta\in J$. Therefore, for each $a\in I$ there exists $t^a_P\in\mathcal{C}_R(P)$ such that $t^a_Pa\in J$ for all maximal ideals $P$ of $R$.

Since $R$ is Noetherian, we can write $I=a_1R+a_2R+\cdots+a_nR$ for some $a_i\in I$. Fix a maximal ideal $P$ of $R$ and consider the element $t_P^{a_1}\in\mathcal{C}_R(P)$. Then 
\[t_P^{a_1}I=t_P^{a_1}a_1R+t_P^{a_1}(a_2R+\cdots+a_nR)\subseteq J+t_P^{a_1}a_2R+\cdots t_P^{a_1}a_nR.\]
Now, set $x_2=t_P^{a_1}a_2\in I$. Consider $t_P^{x_2}\in\mathcal{C}_R(P)$ such that $t_P^{x_2}x_2\in J$. Then,
\begin{equation*}
\begin{split}
t_P^{x_2}t_P^{a_1}I & \subseteq t_P^{x_2}J+t_P^{x_2}t_P^{a_1}a_2R+t_P^{x_2}(t_P^{a_1}a_3R\cdots t_P^{a_1}a_nR) \\
& \subseteq J+t_P^{x_2}t_P^{a_1}a_3R\cdots t_P^{x_2}t_P^{a_1}a_nR.
\end{split}
\end{equation*}
Continuing in this way, we get an element $u_P:=t_P^{x_n}\cdots t_P^{x_2}t_P^{a_1}\in\mathcal{C}_R(P)$ such that $u_PI\subseteq J$. This implies that for every maximal ideal $P$ of $R$, $(J:I)\cap\mathcal{C}_R(P)\neq\emptyset$. Therefore $R=(I:J)$ and so $I\subseteq J$. Symmetrically, $J\subseteq I$. 
\end{proof}

\begin{cor}\label{totalloc}
Let $R$ be a Noetherian ring such that each maximal ideal of $R$ is localizable, and let $I$ be an ideal of $R$. If $I^e=R_P$ for each maximal ideal $P$, then $I=R$.
\end{cor}

\begin{lem}
Let $P$ be a localizable prime ideal of the prime Goldie ring $R$. Then, $R$ is $\mathcal{C}_R(P)$-torsionfree, here $\mathcal{C}_R(P)$ denotes the set of regular elements modulo $P$.
\end{lem}

\begin{proof}
Let $t_P(R)$ denote the $\mathcal{C}(P)$-torsion ideal of $R$. If $0\neq t_P(R)$, then it is essential in $R$ as left (and right) ideal because $R$ is prime. Hence there is a $c\in\mathcal{C}_R(0)$ with $c\in t_P(R)$, since $R$ is Goldie. Therefore, there exists $s\in\mathcal{C}_R(P)$ such that $sc=0$. Since $c$ is regular, $s=0$ which is a contradiction. Thus, $t_P(R)=0$.
\end{proof}

Under the hypothesis of last lemma, we can identify $R$ with its image in $R_P$. Moreover, $\mathcal{C}_R(P)\subseteq\mathcal{C}_R(0)$ and hence $R_P$ is a subring of $Q_{cl}(R)$. We write the following statements for the convenience of the reader.

\begin{thm}[\cite{mcconnell2001noncommutative}, 5.2.6]
Let $R$ be a prime Goldie ring. The following conditions are equivalent:
\begin{enumerate}
\item[(1)] each nonzero submodule of a (left or right) progenerator is a generator;
\item[(2)] $R$ is a maximal order and each ideal is finitely generated projective as a left or right module;
\item[(3)] $R$ is a maximal order and each ideal is reflexive;
\item[(4)] each nonzero ideal of $R$ is invertible.
\end{enumerate}
\end{thm}

\begin{dfn}
A prime Goldie ring satisfying the conditions of last theorem is called an \emph{Asano prime ring}.
\end{dfn}

\begin{thm}[\cite{mcconnell2001noncommutative}, 5.2.9]\label{prodmax}
If $R$ is an Asano prime ring then each nonzero ideal is a unique (commutative) product of maximal ideals.
\end{thm}


\begin{thm}[\cite{chatters1972localisation}, Theorem 1.3]\label{orepinv}
Let $R$ be Noetherian ring and let $P$ be an invertible prime ideal of $R$. Then $R$ satisfies the Ore condition with respect to $\mathcal{C}_R(P)$.
\end{thm}

\begin{cor}\label{primosloc}
Let $R$ be a Noetherian ring. If $R$ is an Asano prime ring then, every prime ideal of $R$ is localizable.
\end{cor}

\begin{proof}
It follows from Theorem \ref{orepinv} and \cite[Proposition 10.7]{goodearl2004introduction}.
\end{proof}

\begin{lem}\label{trasladados}
Let $R$ be such that the product of ideals is commutative. Then, $(I:_\ell J)=(I:_rJ)$ for all ideals $I,J$ of $R$.
\end{lem}

\begin{proof}
Let $I$ and $J$ ideals of $R$. Then $(I:_\ell J)$ is an ideal and $(I:_\ell J)J\subseteq I$. By hypothesis, $J(I:_\ell I)\subseteq I$. Hence, $(I:_\ell J)\subseteq (I:_rJ)$. Analogously, $(I:_rJ)\subseteq (I:_\ell J)$.
\end{proof}

%

\begin{prop}\label{sdmlasano}
Let $R$ be a prime Noetherian ring. Suppose that each ideal $I$ can be written as $Rc_1+\cdots+Rc_m=I=d_1R+\cdots +d_nR$ with $Rc_i$ and $d_iR$ ideals for $1\leq j\leq m$ and $1\leq i\leq n$. If the product of ideals of $R$ is commutative and $R$ satisfies \emph{SDML}, then $R$ is an Asano prime ring.
\end{prop}

\begin{proof}
Let $Rc_1+\cdots+Rc_m=I=d_1R+\cdots +d_nR$ be an ideal of $R$ with $Rc_i$ and $d_iR$ ideals for $1\leq j\leq m$ and $1\leq i\leq n$. 
We proceed by induction on $n$ to prove that $I$ is left invertible. Note that if $I=d_1R$. Since $I$ is an ideal $d_1\in\mathcal{C}_R(0)$. Hence $(Rd_1^{-1})I=(Rd_1^{-1})d_1R=R$. Thus, $I$ is left invertible. Let $0\neq I=d_1R+d_2R$ and denote $J=d_1R$ and $K=d_2R$. By Lemma \ref{trasladados}, $(J:_\ell K)=(J:_rK)$, so we just write $(J:K)$.  Assume $J\neq 0$ and $K\neq 0$. By hypothesis, $R=(J:K)+(K:J)$ and so $1=x+y$ with $x\in(J:K)$ and $y\in (K:J)$. Hence, $xK\subseteq J$ and $yJ\subseteq K$. Therefore, $xK^2\subseteq JK$ and $yJ^2\subseteq KJ$. Thus,
\[(xK+yJ)(K+J)=xK^2+yJK+xKJ+yJ^2\subseteq JK+KJ=JK,\]
since the product of ideals is commutative. On the other hand,
\[JK=xJK+yJK=xKJ+yJK\subseteq (xK+yJ)(K+J).\]
Then, $JK=(xK+yJ)(K+J)$. Since $J$ and $K$ are ideals, $d_1,d_2\in\mathcal{C}_R(0)$. Consider $Rd_2^{-1}Rd_1^{-1}$. Then,
\[(Rd_2^{-1}Rd_1^{-1})JK=(Rd_2^{-1}Rd_1^{-1})d_1Rd_2R=Rd_2^{-1}Rd_2R=Rd_2^{-1}d_2R=R,\]
since $d_2R$ is an ideal. Therefore, $JK$ is left invertible. This implies that $I=J+K$ is left invertible. Now suppose $n\geq 3$ and $I=d_1R+\cdots d_nR$. By induction hypothesis, $L=d_1R+\cdots+d_{n-1}R$ is left invertible, that is, there exists, and $R$-ideal $L^{-1}$ such that $L^{-1}L=R$. Since $d_nR$ is an ideal, $d_n\in\mathcal{C}_R(0)$. As we did above, $d_nRL=(xd_nR+yL)(d_nR+L)$. Then
\[L^{-1}Rd_n^{-1}d_nRL=L^{-1}RL=R.\]
Thus, $I=L+d_nR$ is left invertible. Analogously, by induction over $m$, $I$ is right invertible. Then $I$ is invertible and so $R$ is a prime Asano ring.
\end{proof}

\begin{obs}\label{commhyp}
Note that any commutative Noetherian domain satisfies the general hypothesis of Proposition \ref{sdmlasano} by \cite[Corollary 3.3.7]{mcconnell2001noncommutative}.
\end{obs}

\begin{thm}\label{asano}
Consider the following conditions for a Noetherian prime ring $R$:
\begin{enumerate}
\item[(1)] $R$ is an Asano prime ring;
\item[(2)] Every prime ideal $P$ of $R$ is localizable and each ideal $I\in \Lambda^{fi}(R_P)$ is a power of $J(R_P)$, the Jacobson radical of $R_P$, where $R_P$ denotes the localization of $R$ at $P$. In particular, $\Lambda^{fi}(R_P)$ is totally ordered.
\item[(3)] The product in $\Lambda^{fi}(R)$ is commutative and 
\[(A:B)+(B:A)=R,\] 
for all $A,B\in \Lambda^{fi}(R)$;
\item[(4)] The product in $\Lambda^{fi}(R)$ is commutative and 
\[(A+B):C=(A:C)+(B:C),\]
for all $A,B,C\in \Lambda^{fi}(R)$;
\item[(5)] The product in $\Lambda^{fi}(R)$ is commutative and 
\[A:(B\cap C)=(A:B)+(A:C),\]
for all $A,B,C\in \Lambda^{fi}(R)$;
\end{enumerate}
Then, $(1)\Rightarrow(2)\Rightarrow(3)\Leftrightarrow(4)\Leftrightarrow(5)$. If each ideal $A$ can be written as $Rc_1+\cdots+Rc_m=A=d_1R+\cdots +d_nR$ with $Rc_i$ and $d_iR$ ideals for $1\leq j\leq m$ and $1\leq i\leq n$ then all the conditions are equivalent.
\end{thm}

\begin{proof}
\textit{(1)}$\Rightarrow$\textit{(2)} By Corollary \ref{primosloc}, each prime ideal of $R$ is localizable. Moreover, by Proposition \ref{prodmax} each nonzero prime ideal is maximal. Let $P$ be any prime ideal of $R$. If $P=0$, then the classical ring of quotients $Q$ of $R$ is simple and hence $\Lambda^{fi}(Q)$ is totally ordered. Suppose $P\neq 0$, and so $P$ is maximal. Note that any ideal $A$ such that $A\nsubseteq P$ contains an element regular modulo $P$. Let $A$ and $B$ ideals of $R$. By Proposition \ref{prodmax} there exist natural numbers $\alpha_1,\dots,\alpha_n,\beta_1,\dots,\beta_n$ and distinct maximal ideals of $R$, $\mathcal{M}_1,\dots,\mathcal{M}_n$ such that $A=\mathcal{M}_1^{\alpha_1}\cdots\mathcal{M}_n^{\alpha_n}$ and $B=\mathcal{M}_1^{\beta_1}\cdots\mathcal{M}_n^{\beta_n}$. If $P\neq \mathcal{M}_i$ for all $1\leq i\leq n$, then $\mathcal{C}_R(P)\cap\mathcal{M}_i\neq\emptyset$. This implies that $A^e=R_P=B^e$. So, assume $P=\mathcal{M}_j$ for some $1\leq j\leq n$. Then,
\begin{equation*}
\begin{split}
A^e & =(\mathcal{M}_1^{\alpha_1}\cdots\mathcal{M}_n^{\alpha_n})^e \\
& =(\mathcal{M}_1^{\alpha_1})^e\cdots(\mathcal{M}_n^{\alpha_n})^e \\
& =R_P\cdots R_P(\mathcal{M}_j^{\alpha_j})^eR_P\cdots R_P \\
& =(P^e)^{\alpha_j} \\
& =J(R_P)^{\alpha_j}.
\end{split}
\end{equation*}
Analogously, $B^e=J(R_P)^{\beta_j}$. Since any ideal $I$ of $R_P$ has the form $I=A^e$ for some ideal $A$ of $R$, this proves \textit{(2)}.

\textit{(2)}$\Rightarrow$\textit{(3)} Let $A,B\in\Lambda^{fi}(M)$. Then, for each maximal ideal of $R$, 
\[(AB)^e=A^eB^e=J(R_P)^nJ(R_P)^k=B^eA^e=(BA)^e.\]
By Proposition \ref{idealesiguales}, $AB=BA$. On the other hand, by Lemma \ref{extdospuntos} $(A^e:B^e)\subseteq(A:B)^e$ and $(B^e:A^e)\subseteq (B:A)^e$ for each maximal ideal of $R$. Therefore,
\[(A^e:B^e)+(B^e:A^e)\subseteq (A:B)^e + (B:A)^e=((A:B) + (B:A))^e,\]
for each maximal ideal of $R$. By hypothesis, $A^e\subseteq B^e$ or $B^e\subseteq A^e$ for each maximal ideal of $R$. This implies that $(A^e:B^e)=R_P$ or $(B^e:A^e)=R_P$ for each maximal ideal $P$ of $R$. Hence $((A:B) + (B:A))^e=R_P$ for all maximal ideal $P$ of $R$. By Corollary \ref{totalloc}, $(A:B) + (B:A)=R$.

$(3)\Leftrightarrow(4)\Leftrightarrow(5)$ Follows by \cite[Proposition 3.2]{niefield1995algebraic}. 


\textit{(3)}$\Rightarrow$\textit{(1)} It follows from Proposition \ref{sdmlasano}.
\end{proof}

\begin{cor}
Let $R$ be a Noetherian Asano prime ring. Then, $\Lambda^{fi}(R)$ is distributive. 
\end{cor}

\begin{proof}
By Theorem \ref{asano}, $R$ satisfies SDML2 and by Proposition \ref{sdmls}, $\Lambda^{fi}(R)$ is distributive.
\end{proof}

By Remark \ref{commhyp} and Theorem \ref{asano} we get the following corollary which is part of \cite[Theroem]{niefield1985strong}.

\begin{cor}
The following are equivalent for a commutative Noetherian domain $R$.
\begin{enumerate}
\item[(1)] $R$ is a Dedekind domain;
\item[(2)] $(A:B)+(B:A)=R$, for all $A,B\in \Lambda(R)$;
\item[(3)] $(A+B):C=(A:C)+(B:C)$, for all $A,B,C\in \Lambda(R)$;
\item[(4)] $A:(B\cap C)=(A:B)+(A:C)$, for all $A,B,C\in \Lambda(R)$.
\item[(5)] $\Lambda(R_P)$ is totally ordered, for every prime ideal $P$ of $R$, where $R_P$ denotes the localization of $R$ at $P$.
\end{enumerate}
\end{cor}

\section{Acknowledgements}
Part of this investigation was made on a visit of the first author to the Universidad de Guadalajara, he wishes to thank the members of the Department of Mathematics for their kind hospitality. This visit was supported by the program PROSNI 2019 of the third author. The second author thanks  project  {\it PRODEP PTC-2019 grant UAM-PTC-700 Num. 12613411 awarded by SEP.} 

The documents \cite{simmonsintroduction}, \cite{simmonsvarious} that we mentioned in the manuscript were available on the author's personal web page. Unfortunately these references are not available anymore.
\bibliographystyle{amsalpha}
\bibliography{BNota}

\scriptsize
 \vskip5mm
\begin{itemize}\item	 \textit{Mauricio Medina-B\'arcenas$^{\ast}$, }
Facultad de Ciencias F\'isico-Matem\'{a}ticas, Benem\'erita Universidad Aut\'{o}noma de Puebla. Av. San Claudio y 18 Sur, Col. San Manuel, Ciudad Universitaria, 72570, Puebla, M\'exico

\item \textit{Martha Lizbeth Shaid Sandoval Miranda$^{(\ddag)}$ }
Departamento de Matem\'aticas, Universidad Aut\'{o}noma Metropolitana - Unidad Iztapalapa, Av. San Rafael Atlixco 186, Col. Vicentina
Iztapalapa, 09340, Ciudad de M\'exico, M\'exico. 
\item \textit{\'Angel Zald\'ivar Corichi$^{\dag}$.}
Departamento de Matem\'aticas, Centro Universitario de Ciencias Exactas e Ingenier\'ias, Universidad de Guadalajara, Blvd. Marcelino Garc\'ia Barrag\'an, 44430, Guadalajara, Jalisco, M\'exico.\vspace{5pt}
\end{itemize}
\begin{itemize}
\item[\hspace{33pt}] $\ast$ mmedina@fcfm.buap.mx
\item[\hspace{33pt}] $\ddag$ marlisha@xanum.uam.mx
\item[\hspace{33pt}]  $\dag$ luis.zaldivar@academicos.udg.mx (Corresponding author).
\end{itemize}

\end{document}